\documentclass[12pt]{article}

\usepackage{babel}
\usepackage[font=small,format=plain,labelfont=bf,up,textfont=it,up]{caption}
\usepackage{mathrsfs,diagbox}
\usepackage{amsmath,amsfonts,amssymb,amsthm}
\usepackage{ucs} 
\usepackage[T1]{fontenc} 
\usepackage{enumerate}
\usepackage{textcomp,url}
\usepackage{eucal}
\usepackage{comment}
\usepackage[noBBpl]{mathpazo}
\usepackage[matrix,arrow]{xy}
\usepackage{epsfig}
\usepackage{multicol}

\usepackage[pdftex,                %
implicit=true,
bookmarks=true,
breaklinks=true,
bookmarksnumbered = true,
colorlinks= true,
urlcolor= green,
anchorcolor = yellow,
citecolor=blue,
pdfborder= {0 0 0}
]{hyperref}

\usepackage{ytableau} 
\usepackage{multicol}

\makeatletter

\renewcommand{\p@enumii}{}
\def\@enum@{\list{\csname label\@enumctr\endcsname}%
{\usecounter{\@enumctr}\def\makelabel##1{
\normalfont\ignorespaces\emph{{##1}~}}
\setlength{\labelsep}{3pt}
\setlength{\parsep}{0pt}
\setlength{\itemsep}{0pt}
\setlength{\leftmargin}{0pt}
\setlength{\labelwidth}{0pt}
\setlength{\listparindent}{\parindent}
\setlength{\itemsep}{0pt}
\setlength{\itemindent}{0pt}
\topsep=3pt plus 1pt minus 1 pt}}
\makeatother

\renewcommand{\epsilon}{\ensuremath{\varepsilon}}
\renewcommand{\phi}{\ensuremath{\varphi}}

\renewcommand{\to}{\ensuremath{\longrightarrow}}
\renewcommand{\mapsto}{\ensuremath{\longmapsto}}

\newcommand{\R}{\ensuremath{\mathbb R}}

\newcommand{\Z}{\ensuremath{\mathbb Z}}

\newcommand{\St}[1][2]{\ensuremath{\mathbb S}^{#1}}

\renewcommand{\ker}[1]{\ensuremath{\operatorname{\text{Ker}}\left({#1}\right)}}

\makeatletter
\def\@map#1#2[#3]{\mbox{$#1 \colon\thinspace #2 \to #3$}}
\def\map#1#2{\@ifnextchar [{\@map{#1}{#2}}{\@map{#1}{#2}[#2]}}
\makeatother

\newcommand{\ang}[1]{\ensuremath{\left\langle #1\right\rangle}}

\newcommand{\setang}[2]{\ensuremath{\ang{#1 \,\mid\, #2}}}

\newtheoremstyle{theoremm}{}{}{\itshape}{}{\scshape}{.}{ }{}
\theoremstyle{theoremm}

\newtheorem{thm}{Theorem}
\newtheorem{lem}[thm]{Lemma}
\newtheorem{prop}[thm]{Proposition}
\newtheorem{cor}[thm]{Corollary}

\newtheoremstyle{remark}{}{}{}{}{\scshape}{.}{ }{}
\theoremstyle{remark}

\newtheorem{defn}[thm]{Definition}
\newtheorem{rem}[thm]{Remark}
\newtheorem{rems}[thm]{Remarks}

\newtheoremstyle{comment}{}{}{\bfseries}{}{\bfseries}{:}{ }{}
\theoremstyle{comment}

\newcommand{\redefn}[1]{Definition~\protect\ref{defn:#1}}
\newcommand{\rethm}[1]{Theorem~\protect\ref{thm:#1}}
\newcommand{\relem}[1]{Lemma~\protect\ref{lem:#1}}
\newcommand{\reprop}[1]{Proposition~\protect\ref{prop:#1}}

\newcommand{\rerem}[1]{Remark~\protect\ref{rem:#1}}

\newcommand{\reqref}[1]{(\protect\ref{eq:#1})}
\allowdisplaybreaks

\newcommand{\tq}{\,|\,}
\newcommand{\apres}[2]{\langle #1 \tq #2\rangle}

\newcommand{\BZ}{\mathbb{Z}}

\newcommand{\ie}{\emph{i.e.}}

\begin{document}

\title{Subgroup separability of surface braid groups and virtual braid
groups}


\author{KISNNEY ALMEIDA\\
Universidade Estadual de Feira de Santana,\\
Departamento de Ciências Exatas,\\
Av. Transnordestina S/N, CEP 44036-900 - Feira de Santana - BA - Brazil.\\
e-mail:~\texttt{kisnney@gmail.com}\vspace*{4mm}\\
IGOR~LIMA\\
Universidade de Bras\'ilia,\\
Departamento de Matem\'atica, \\
Bras\'ilia, DF, 70910-900, Brasil.\\
e-mail:~\texttt{igor.matematico@gmail.com}\vspace*{4mm}\\
OSCAR~OCAMPO~\\
Universidade Federal da Bahia,\\
Departamento de Matem\'atica - IME,\\
Av. Adhemar de Barros~S/N~CEP:~40170-110 - Salvador - BA - Brazil.\\
e-mail:~\texttt{oscaro@ufba.br}
}


\maketitle

\begin{abstract}
In this paper, we study subgroup separability (also known as LERF) and related properties for surface braid groups and virtual (singular) braid groups
 \end{abstract}

\section{Introduction}\label{intro}


The braid groups of the $2$-disc, or Artin braid groups, were introduced by Artin in 1925 and further studied in 1947~\cite{A1,A2}. Surface braid groups were initially studied by Zariski~\cite{Z}, and were later generalised by Fox and Neuwirth to braid groups of arbitrary topological spaces using configuration spaces as follows~\cite{FoN}. Let $S$ be a compact, connected surface, and let $n\in \mathbb N$. The \textit{$n$th ordered configuration space of $S$}, denoted by $F_{n}(S)$, is defined by:
\begin{equation*}
F_n(S)=\left\{(x_{1},\ldots,x_{n})\in S^{n} \mid x_{i}\neq x_{j}\,\, \text{if}\,\, i\neq j;\,i,j=1,\ldots,n\right\}.
\end{equation*}
The \textit{$n$-string pure braid group $P_n(S)$ of $S$} is defined by $P_n(S)=\pi_1(F_n(S))$. The symmetric group $S_{n}$ on $n$ letters acts freely on $F_{n}(S)$ by permuting coordinates, and the \textit{$n$-string braid group $B_n(S)$ of $S$} is defined by $B_n(S)=\pi_1(F_n(S)/S_{n})$. This gives rise to the following short exact sequence:
\begin{equation}\label{eq:ses}
1 \to P_{n}(S) \to B_{n}(S) \stackrel{\sigma}{\longrightarrow} S_{n} \to 1.
\end{equation}
The map $\map{\sigma}{B_{n}(S)}[S_{n}]$ is the standard homomorphism that associates a permutation to each element of $S_{n}$.


\begin{rems}\label{rem:pi1}
    \begin{enumerate}
        \item     Follows from the definition that $F_1(S)=S$ for any surface $S$, the groups $P_1(S)$ and $B_1(S)$ are isomorphic to $\pi_1(S)$.
        For this reason, braid groups over the surface $S$ may be seen as generalizations of the fundamental group of $S$.

        \item Let $S$ be a surface with boundary $\partial S$ and let $S'=S\setminus \partial S$.
We note that the $n^{th}$ (pure) braid groups  of $S$ and $S'$ are isomorphic, see
\cite[Remarks~8(d)]{GPi}.

    \end{enumerate}

\end{rems}

For more information on general aspects of surface braid groups we recommend the survey \cite{GPi}, in particular its Section~2 where equivalent definitions of these groups are given, showing different viewpoints of these groups.

A group $G$ is called \textit{subgroup separable} or \textit{locally extended
residually finite (LERF)} if each f.g. subgroup $H$ of $G$ is the intersection of finite index subgroups of $G$. A group $G$ is called \textit{residually finite} if the trivial subgroup is the intersection of the finite index
subgroups of $G$. The concept of subgroup separability clearly implies residual finiteness and it was first studied by Hall \cite{H}. Mal'cev proved residual finiteness implies solvable word problem \cite{Mal} and the same idea may be adapted to prove LERF implies solubility of the generalized word problem, as explained in \cite{FW}.




Subgroup separability is also connected to the profinite topology of $G$, \ie ,  the topo\-logy in which a base for the open sets is the set of all cosets of normal subgroups of finite index in $G$. A group $G$ is subgroup separable if and only if every finitely generated subgroup of $G$ is closed in the profinite topology \cite{H2}.

Usually it is hard to establish general results on subgroup separability, but we know that finite groups, abelian groups \cite{Mal}, surface groups \cite{S2}, free groups \cite{H} and free products of LERF groups are LERF \cite{Bu}, for example.

The case of classical braid groups has already been established \cite{DM}, so it is natural to investigate the available generalizations. In a previous work, Almeida and Lima have established a criterion for the subgroup separability of Artin groups \cite{AL}, which inspired us to investigate the case of surface braid groups and virtual braid groups, as we do in this paper. We also discuss some properties that are related to subgroup separability for those groups.


The paper is organized as follows: In Section \ref{sec:prelim}, we explain some facts on subgroup separability we will need. In Section \ref{sec:sbg} we deal with  surface braid groups: In the first three subsections we present our main tools - the analysis of which surface braid groups contain $F_2\times F_2$ as a subgroup, the Fadell-Neuwirth short exact sequence and property (LR), respectively - and in the fourth subsection we prove our main results on subgroup separability of surface braid groups and other consequences. Finally in Section \ref{sec:vbg} we study the case of virtual braid groups.

\subsection*{Acknowledgments}
The first author was partially supported by FINAPESQ-UEFS, FAPDF, Brazil. The second author was partially supported by DPI/UnB, FAPDF, Brazil. The third author was partially supported by FAPDF, Brazil and by National Council for Scientific and Technological Development (CNPq) through a \textit{Bolsa de Produtividade} 305422/2022-7.

\section{Subgroup separability}\label{sec:prelim}

In this section we establish some results on subgroup separability that we will need. First we present its behaviour on subgroups.

\begin{thm}[{\cite[Lemma~1.1]{S}}]\label{thm:lerf1}
Let $H$ be a subgroup of $G$. If $H$ is not LERF then $G$ is not LERF either. The converse holds if $[G:H]< \infty$.
\end{thm}

The disk case is already fully solved.

\begin{thm}[{\cite{DM}}]\label{thm:disk}
If $S$ is the disk, $B_n(S)$ is LERF if and only if $n\leq 3$.
\end{thm}

When investigating subgroup separability of braids, it is enough to deal with the pure case, since they form a finite index subgroup of the full group.

\begin{cor}\label{cor:braidpure}
The (full) surface braid group $B_n(S)$ (resp. the virtual braid group $VB_n$) is LERF if and only if the pure surface braid group $P_n(S)$ (resp. the pure virtual braid group $VP_n$) is.
\end{cor}

\begin{proof}
It follows from the remark above and Theorem \ref{thm:lerf1}.
\end{proof}

\begin{cor}\label{cor:puredisk}
If $S$ is the disk, $P_n(S)$ is LERF if and only if $n\leq 3$.
\end{cor}

\begin{proof}
It follows from Theorem \ref{thm:disk} and Corollary \ref{cor:braidpure}.
\end{proof}

Althought the direct product of two LERF groups is not always LERF, it is true for some specific  direct products - and free products.

\begin{thm}[{\cite[Corollary~1.2]{Bu}}]\label{thm:lerf2}
A free product of LERF groups is LERF.
\end{thm}

\begin{lem}[{\cite[Lemma~3.6]{AL}}]\label{lem:lerf1}

Let $G$ be a group. If $G$ is LERF then $G\times \mathbb{Z}$ is LERF.
\end{lem}

Theorem \ref{thm:lerf1} gives us the possibility of giving a negative answer for subgroup separabi\-lity of a group by analyzing its subgroups. There are two well-known groups that have been used for this purpose, which we present below.

\begin{lem}
The group $F_2\times F_2$ is not LERF
\end{lem}
\begin{proof}
The group $F_2\times F_2$ has unsolvable generalized word problem \cite{Mik2} hence it may not be LERF - see Section \ref{gwp} for more details.
\end{proof}

\begin{lem}[{\cite{BKS}}]\label{lem:Knotlerf}
The group $K=\apres{y, \alpha, \beta}{\alpha ^y=\beta\alpha, \beta^y=\beta}$ is not LERF.
\end{lem}

We shall use the following notation: $g^h = h^{-1}gh$. Then from \cite[Equation~(1)]{BKS} the group $K$ has generators $y, \alpha, \beta$ and relations $y^{-1}\alpha y = \alpha\beta$ and $y^{-1}\beta y = \beta$.   Taking the inverse of both relations we obtain $y^{-1}\alpha ^{-1} y = \beta^{-1}\alpha^{-1}$ and $y^{-1}\beta^{-1} y = \beta^{-1}$. Then, taking $\alpha_1=\alpha^{-1}$ and $\beta_1=\beta^{-1}$ we obtain the following presentation for $K$:
$$
K=\apres{y, \alpha_1, \beta_1}{y^{-1}\alpha_1 y = \beta_1\alpha_1, y^{-1}\beta_1 y = \beta_1}.
$$
Since $y$ commutes with $\beta_1$ then we may rewrite the first relation as $y \alpha_1 y^{-1} = \beta_1^{-1}\alpha_1$.
Then returning to $\beta$ (in fact, we are using the automorphism $y\mapsto y$, $\alpha_1\mapsto \alpha_1$ and $\beta^{-1}\mapsto \beta$) we obtain the following presentation for $K$:
\begin{equation}\label{eq:k}
K=\apres{y, \alpha_1, \beta}{y \alpha_1 y^{-1}=\beta\alpha_1, y \beta y^{-1}= \beta}.
\end{equation}

\section{Surface braid groups}\label{sec:sbg}

We start this section by establishing the language that we shall use repeatedly in this paper.
As in \cite{PR}, a compact surface $S$ will be called \emph{large} if it is different from
\begin{multicols}{3}
\begin{itemize}
	\item the sphere,
	\item the projective plane,
	\item the disk,
	\item the annulus,
	\item the torus,
	\item the Möbius strip, or
	\item the Klein bottle.
\end{itemize}
\end{multicols}
We shall call these seven surfaces \emph{non-large surfaces}.
From \cite[Proposition~1.6]{PR} follows that if $S$ is a large compact surface then the center  of the (pure) braid group ($P_n(S)$) $B_n(S)$ is the trivial group.

Let $Pants$ denote the sphere minus 3 points.
\begin{figure}[!htb]
\centering
\includegraphics[scale=0.5]{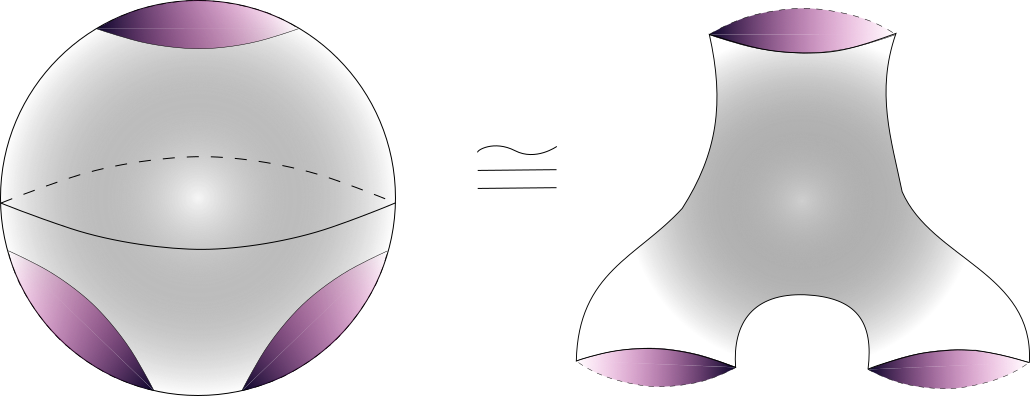}
\caption{The sphere minus 3 points}
\label{pant}
\end{figure}

\begin{rem}
We recall that a \textit{subsurface} $N$ of a surface $M$ is the closure of an open subset of $M$.
Let $x\in N$. Let $N$ be a subsurface of $M$ such that $\pi_1(N,x)\neq \{1\}$.
The inclusion $N\subseteq M$ induces a homomorphism $\psi\colon \pi_1(N,x)\to \pi_1(M,x)$ that is injective if and only if none of the connected components of the closure $\overline{M\setminus N}$ of $M\setminus N$ is a disk, see \cite[Proposition~2.1]{PR}.
Using this result, we note that the fundamental group of large surfaces contains a copy of the rank two free group, since the pants surface is a subsurface of them and $\pi_1(Pants)$ is a free group of rank 2.
\end{rem}

\subsection{The Fadell-Neuwirth short exact sequence}\label{ss:fadell}

Let $S$ be a connected surface and let $n\in \mathbb{N}$.
If $m\geq 1$, the map $p\colon F_{n+m}(S)\to F_n(S)$, of the configuration space $F_{n+m}(S)$ onto $F_n(S)$, defined by $p(x_1,\ldots,x_n,\ldots,x_{n+m}) = (x_1,\ldots,x_n)$ induces a homomorphism $p_{\ast}\colon P_{n+m}(S)\to P_n(S)$.
The homomorphism $p_{\ast}$ geometrically ``forgets'' the last $m$ strings.
If $M$ is without boundary, Fadell and Neuwirth showed that $p$ is a locally-trivial fibration \cite[Theorem~1]{FaN}, with fibre $F_m(M\setminus \{ x_1,\ldots,x_n \})$ over the point $(x_1,\ldots,x_n)$, which we consider to be a subspace of the total space via the map $i\colon F_m(M\setminus \{x_1,\ldots,x_n\})\to F_{n+m}(M)$ defined by $i((y_1,\ldots,y_m)) = (x_1,\ldots,x_n,y_1,\ldots,y_m)$.
Applying the associated long exact sequence in homotopy to this fibration, we obtain the Fadell-Neuwirth short exact sequence of pure braid groups:
\begin{equation}\label{eq:ses}
1\to P_{m}(S\setminus \{x_1,\ldots,x_n\}) \stackrel{i_{\ast}}{\longrightarrow} P_{n+m}(S) \stackrel{p_{\ast}}{\longrightarrow} P_n(S) \to 1
\end{equation}
where $n\geq 3$ if $S$ is the sphere \cite{F, FvB}, $n\geq 2$ if $S$ is the real projective plane \cite{vB}, and $n\geq 1$ otherwise \cite{FaN}, and $i_{\ast}$ is the homomorphism induced by the map $i$.
This sequence has been widely studied.
For instance, one question studied for many authors during several years was the splitting problem for surface pure braid groups: does the short exact sequence \reqref{ses} split?
The latter is completely solved, see \cite{GG} for more details, in particular its Theorem~2. Additional information on this sequence may be seen in \cite[Section~3.1]{GPi}.

In the following remarks, we record explicit information of some surface braid groups that we will use several times during the text. In many of them, it is a direct product decomposition using the splitting of the Fadell-Neuwirth short exact sequence with trivial action, since there is a section that sends generators of the quotient group into central elements of the respective surface braid group, see \cite[Theorem~2]{GG} and the references therein.
The free group of rank 2 will be denoted by $F_2$.

\begin{rems}\label{rems:sbg}
Suppose $n\geq 1$.

\begin{enumerate}
    \item Braid groups with few strings over the sphere and the projective plane are finite:
    \begin{itemize}
        \item $B_1(\St[2])$, $P_1(\St[2])$ and $P_2(\St[2])$ are trivial groups,  $B_2(\St[2])\cong \Z_2$, $B_3(\St[2])$ is isomorphic to $\Z_3\rtimes \Z_4$ with non-trivial action and $P_3(\St[2])\cong \Z_2$, see \cite{FvB} and also \cite[Section~4]{GPi}.

        \item $B_1(\R P^2)=P_1(\R P^2)=\pi_1(\R P^2)\cong \Z_2$, $B_2(\R P^2)$ has order 16 and is isomorphic to the generalized quaternions,and $P_2(\R P^2)$ is isomorphic to the quaternion group, see \cite{vB} and also \cite[Section~4]{GPi}.
    \end{itemize}

\item If $S$ is the sphere, then $P_{n+3}(S)\cong P_n(Pants)\times \BZ_2$.
In particular, $P_4(S)\cong P_1(Pants)\times \BZ_2\cong F_2\times \BZ_2$.

\item Suppose $S$ is the disk. It is an immediate consequence of the classical Artin presentation of $P_2(S)$ and $B_2(S)$ that they are isomorphic to $\BZ$, see \cite{A2} and also \cite{KM}.
Let $n\geq 3$. There is a decomposition $P_{n+2}(S)\cong P_n(Pants)\times \BZ$ that follows from the splitting of the Fadell-Neuwirth short exact sequence.
In particular, note that $P_3(S)\cong F_2\times \BZ$.

\item If $S$ is the annulus then $P_{n+1}(S)\cong P_n(Pants)\times \BZ$.
In particular, $P_2(S)\cong P_1(Pants)\times \BZ\cong F_2\times \BZ$.

\item If $S$ is the torus then $P_{n+1}(T)\cong P_n(T\setminus \{x_1\})\times \BZ^2$.
In particular, $P_2(T)\cong P_1(T\setminus \{x_1\})\times \BZ^2\cong F_2\times \BZ^2$.
\end{enumerate}

\end{rems}

\subsection{The direct product of free groups $F_2 \times F_2$ as a subgroup of surface braid groups}\label{ss:f2xf2}

Let us divide the connected, compact surfaces into three families.

\begin{itemize}
	\item[$\mathcal{F}_1$:] The seven non-large surfaces.

\item[$\mathcal{F}_2$:]
\begin{itemize}
	\item[i)] The pants surface,
	\item[ii)] The torus minus one point,
	\item[iii)] The projective plane minus two points,
    \item[iv)] The Klein bottle minus one point,
    \item[v)] The connected sum of 3 projective planes.
\end{itemize}
	
\item[$ \mathcal{F}_3$:] The connected, compact surfaces not considered in the families $\mathcal{F}_1$ and $\mathcal{F}_2$.

\end{itemize}

We start with the following useful result.

\begin{lem}\label{lem:pant}
Let $n\geq 1$.

\begin{enumerate}
\item If $S$ is a large surface (i.e. $S\in \mathcal{F}_2\cup \mathcal{F}_3$), then $P_n(Pants)$ is a subgroup of $P_n(S)$.

\item Let $X$ be a sphere minus four points.
If $S$ belongs to the family $ \mathcal{F}_3$, then $P_n(X)$ is a subgroup of $P_n(S)$.

\end{enumerate}

\end{lem}

\begin{proof}

First, we note that for item (a) none of the connected components of the closure $\overline{S\setminus Pants}$ of $S\setminus Pants$ is a disk assuming that $S$ is a large surface.
Similarly, for item (b), none of the connected components of the closure $\overline{S\setminus X}$ is a disk, if $S$ belongs to the family $ \mathcal{F}_3$.
See Figure~\ref{pant} for an illustration of the cases of the torus  minus one point and the Klein bottle minus one point.

\begin{figure}[!htb]
\centering
\includegraphics[scale=0.5]{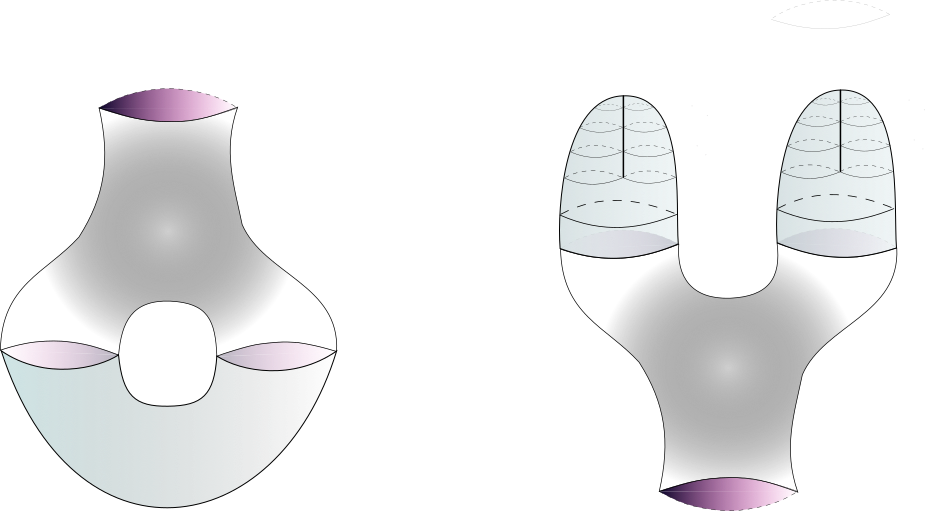}
\caption{The torus minus a point and the Klein bottle minus a point}
\label{pant}
\end{figure}

Hence, from \cite[Proposition~2.2]{PR} and its proof, it follows that $P_n(Pants)$ is a subgroup of $P_n(S)$ proving item (a) and also we conclude that $P_n(X)$ is a subgroup of $P_n(S)$ proving item (b).

\end{proof}

It is well known that the Cartesian product $F_2 \times F_2$ of two copies of the rank two free group is not a subgroup of surface groups. Next we study this problem for surface braid groups.
First, we start with the case of few strands for some specific surfaces.

\begin{prop}
   The Cartesian product $F_2 \times F_2$ of two copies of the rank two free group is not a subgroup of the surface braid groups:
   \begin{enumerate}
       \item $P_2(Pants)$,
       \item $P_n(S)$, where $S$ is the sphere and $n\leq 3$ or $S$ is the projective plane and $n\leq 2$.
   \end{enumerate}
\end{prop}

\begin{proof}
    \begin{enumerate}
       \item We consider the Fadell-Neuwirth short exact sequence
       $$
1\to P_2(Pants) \to P_{4}(D^2) \to P_2(D^2) \to 1.
$$
If $F_2 \times F_2$ is a subgroup of $P_2(Pants)$ then it is a subgroup of $P_4(D^2)$ and also a subgroup of $B_4(D^2)$. But this is a contradiction with the main theorem of \cite{Ak}, who showed that $F_2 \times F_2$ is not a subgroup of $B_4(D^2)$.

       \item  This item is obvious since the groups considered are finite.
   \end{enumerate}
\end{proof}

\begin{rem}\label{rem:maka}

We recall that Makanina constructed a subgroup $F_2(a,b) \times F_2(c,d)$ of the Artin braid group $B_n$, for $n\geq 5$, taking the following elements using the classical Artin presentation of $B_n$: $a=\sigma_3^2$, $b=\sigma_3\sigma_2^2\sigma_3$, $c=\sigma_4\sigma_3\sigma_2^2\sigma_3\sigma_4$ and $d=\sigma_4\sigma_3\sigma_2\sigma_1^2\sigma_2\sigma_3\sigma_4$, see  \cite[Proof of Theorem~1]{Mak}.
\end{rem}

Now, we move for the cases in which $F_2 \times F_2$ is a subgroup of surface braid groups.

\begin{thm}\label{thm:f2f2}
The direct product $F_2 \times F_2$ of two copies of the rank two free group is a subgroup of $P_n(S)$ if one of the following conditions is satisfied:
\begin{enumerate}
\item $S$ is a surface in  $\mathcal{F}_3$ and $n\geq 2$,

\item the surface $S$ belongs to $\mathcal{F}_2$ and $n\geq 3$,

\item $S$ belongs to $\mathcal{F}_1$, but different from the sphere, and $n\geq 5$,

\item $S$ is the sphere and $n\geq 6$.


\end{enumerate}

\end{thm}

\begin{proof}

We shall use the Fadell-Neuwirth short exact sequence \eqref{eq:ses}
$$
1\to P_{m}(S\setminus \{x_1,\ldots,x_n\}) \to P_{n+m}(S) \to P_n(S) \to 1.
$$
for $n\geq 3$ if $S$ is the sphere, for $n\geq 2$ if $S$ is the projective plane, and for $n\geq 1$ otherwise.

Let $n\geq 2$.

\begin{enumerate}
\item Let $X$ be the sphere minus four points. From \relem{pant} if a surface $S$ belongs to the family $ \mathcal{F}_3$ then $P_n(X)$ is a subgroup of $P_n(S)$.
So it is enough to prove $F_2\times F_2$ is a subgroup of $P_n(X)$.

We consider the Fadell-Neuwirth short exact sequence
$$
1\to P_2(X) \to P_{5}(D^2) \stackrel{\phi}{\longrightarrow} P_3(D^2) \to 1
$$
forgetting the last 2 strings of pure braids in $P_5(D^2)$.
Note that the elements $a$, $b$, $c$ and $d$ defined by Makanina (see \rerem{maka}) belong to the kernel of $\phi$, i.e., $F_2 \times F_2$ is a subgroup of $P_2(X)$. Since $P_2(X)$ is a subgroup of $P_n(X)$, for all $n\geq 2$, then the result follows.

\item Let $S$ be a surface in  $\mathcal{F}_2$.
To prove that if $n\geq 3$ then $F_2 \times F_2$ is a subgroup of $P_n(S)$ we may use a similar argument of the previous item, with Pants instead of $X$, but using in this case the Fadell-Neuwirth short exact sequence
$$
1\to P_2(Pants) \to P_{5}(D^2) \stackrel{\psi}{\longrightarrow} P_2(D^2) \to 1
$$
forgetting the last 3 strings of pure braids in $P_5(D^2)$.

\item Let $n\geq 5$. If $S$ is the disk the claim was proved by Makanina \cite{Mak}.
Suppose that $S$ is either the annulus, the torus, the M\"obius strip, or the Klein bottle.
Since the pure Artin braid group $P_n(D^2)$ is a subgroup of $P_n(S)$ (from \cite[Proposition~2.2]{PR} and its proof) then the result of this item follows.

Let $S$ be the projective plane and $n\geq 5$.
In this case we use item (b) of this theorem and the Fadell-Neuwirth short exact sequence
$$
1\to P_{n-2}(X) \to P_{n}(S) \to P_2(S) \to 1,
$$
where $X$ is the projective plane minus two points, to conclude the result.

\item Let $S$ be the sphere and let $n\geq 6$.
Using the Fadell-Neuwirth short exact sequence
$$
1\to P_{n-3}(Pants) \to P_{n}(S) \to P_3(S) \to 1
$$
and item (b) of this theorem we conclude the result for the sphere.
\end{enumerate}
\end{proof}

\subsection{Property (LR)}\label{ss:lr}

A subgroup $H$ of a group $G$ is called a \textit{retract} if there is a homomorphism
$\rho: G \to H$ which restricts to the identity map on $H$. This is equivalent to $G$ splitting as a semidirect product $N$ of $H$, where $N=\ker\rho$. In this case the map $\rho$ is called a \textit{retraction} of $G$ onto $H$.

Let $G$ be a group and let $H$ be a subgroup of G. We will say that $H$ is a \textit{virtual
retract} of $G$, denoted $H \leq_{vr} G$, if there exists a subgroup $K\leq G$ such that $|G:K|<\infty$, $H \subset K$ and $H$ is a retract of K.

If $G$ is a group, we say $G$ has \textit{property (LR) (local retractions)} if all finitely generated subgroups of $G$ are virtual retracts. This property was defined by Long and Reid \cite{LR} although it has been studied long before its explicit definition.

A virtual retract of a residually finite group is closed in the profinite topology \cite{Mi}, hence property (LR) implies subgroup separability. In fact, Scott \cite{S2} proved that all surface groups are LERF essentially by showing that they satisfy property (LR).

Every finite group is trivially (LR), since every subgroup has finite index. Every free group is (LR) \cite{H,Bu2} and (LR) is preserved by free products \cite{Gr}.

Although our main objective in this paper is to establish subgroup separability for braid groups, we can actually prove property (LR) for some of the pure braid groups, which is a stronger result.

 Property (LR) is not preserved by direct products - indeed $F_2\times F_2$ is not (LR) since it is not even LERF - however that is true if  one of the factors is virtually abelian. We recall a virtually abelian group is a group which contains an abelian subgroup of  finite index.

\begin{thm}[{\cite[Proposition~5.6]{Mi}}]\label{thm:Mi1}
Let $X$ be a finitely generated virtually abelian group and $Y$ a group satisfying (LR). Then $X\times Y$ satisfies (LR).
\end{thm}




Now we can deal with the cases below, by using some information of surface braid groups that was collected in Remarks~\ref{rems:sbg}.

\begin{thm}\label{thm:LRnlarge}
Let $S$ be a non-large surface.
\begin{enumerate}
	\item $P_2(S)$ is (LR) if $S$ is the disk, sphere, projective plane, annulus or torus. 
	\item $P_3(S)$ is (LR) if $S$ is the disk or the sphere.

 \item $P_4(S)$ is (LR) if $S$ is the sphere.

\end{enumerate}
\end{thm}

\begin{proof}


\begin{enumerate}
\item
If $S$ is the disk then $P_2(S)=\BZ$ hence it is (LR).

If $S$ is the sphere or the projective plane then $P_2(S)$ is finite hence (LR).

If $S$ is the annulus then $P_2(S)=P_1(Pants)\times \BZ\cong F_2\times \BZ$ hence (LR) by Theorem~\ref{thm:Mi1}.

If $S$ is the torus then $P_2(T)=P_1(T\setminus \{x_1\})\times \BZ^2\cong F_2\times \BZ^2$ hence (LR) by Theorem~\ref{thm:Mi1}.

\item
If $S$ is the disk then $P_3(S)=F_2\times \BZ$ hence (LR) by Theorem~\ref{thm:Mi1}.

If $S$ is the sphere, then $P_3(S)$ is finite hence (LR).

\item
If $S$ is the sphere, then $P_4(S)=P_1(Pants)\times \BZ_2\cong F_2\times \BZ_2$ hence (LR) by Theorem~\ref{thm:Mi1}.
\end{enumerate}

\end{proof}

\subsection{Subgroup separability for surface braid groups}\label{ss:ss-for-sbg}

In this subsection, we determine completely  under which conditions surface braid groups are (or are not) LERF.
First, we consider the general case of large surfaces, then we study case by case the non-large surfaces. During this subsection we shall again use some information of surface braid groups that was collected in Remarks~\ref{rems:sbg}.

\begin{thm}\label{thm:large}

Let $S$ be a large surface and let $n\geq 2$. Then $P_n(S)$ is not LERF.
\end{thm}

\begin{proof}

Let $n\geq 2$ and let $S$ be a large surface. We note that the result follows for some surfaces by a simple application of Theorem~\ref{thm:f2f2}. We shall give now a proof that covers all surfaces of the statement. First, we prove  that the pure surface braid group $P_n(Pants)$ is not LERF.
From the Fadell-Neuwirth short exact sequence \eqref{eq:ses} we obtain the decomposition $P_{n+2}=P_n(Pants)\times \mathbb{Z}$.
If $P_n(Pants)$ is LERF then from Lemma~\ref{lem:lerf1} also $P_{n+2}$ is. But this is a contradiction, since $P_{n+2}$ is not LERF  for $n\geq 2$, by Corollary~\ref{cor:puredisk}.


From Lemma~\ref{lem:pant} follows that $P_n(Pants)$ is a subgroup of $P_n(S)$.
By applying Theorem~\ref{thm:lerf1} we conclude  $P_n(S)$ is not LERF.
\end{proof}

We move to the case of non-large surfaces. We start considering the special case of the Klein bottle.

\begin{prop}\label{prop:kb}
The pure braid group of the Klein bottle with 2 strings $P_2(Kb)$ is not LERF.
\end{prop}

\begin{proof}

We consider here the presentation of the braid groups of the Klein bottle given in \cite[Theorem~2.1]{GP}.
Also, we consider the section $s\colon P_n(Kb)\to P_{n+1}(Kb)$ of the epimorphism $P_{n+1}(Kb)\to P_n(Kb)$, that  geometrically forgets the last string of braids in $P_{n+1}(Kb)$, described in \cite[Proposition~5.1]{GP}.

By the results above, we may assume $P_1(Kb) = \langle a_1, b_1 \mid b_1a_1b_1^{-1} = a_1^{-1} \rangle$.

From \cite[p.~18]{GP} we have  $P_2(Kb)=F_2(a_2,b_2)\rtimes s(P_1(Kb))$, where
$$
s(P_1(Kb)) = \langle a_1a_2, b_2b_1 \mid (b_2b_1)(a_1a_2)(b_2b_1)^{-1} = (a_2a_1)^{-1} \rangle
$$
and the action (see \cite[eqn~(5.9)]{GP}) is given by
\begin{itemize}
    \item $a_1a_2\cdot a_2\cdot (a_1a_2)^{-1} = a_2$,
    \item $a_1a_2\cdot b_2\cdot (a_1a_2)^{-1} = a_2^{-2}b_2$,
    \item $b_2b_1\cdot a_2\cdot (b_2b_1)^{-1} = a_2^{-1}$,
    \item $b_2b_1\cdot b_2\cdot (b_2b_1)^{-1} = a_2b_2a_2$.
\end{itemize}

Let $p\colon P_2(Kb)\to P_1(Kb)$ denote the epimorphism that  geometrically forgets the last string of braids in $P_2(Kb)$. The kernel of $p$ is $F_2(a_2,b_2)$ the free group of rank 2.
Let $Q = \Z \times \Z$ be the finite index 2 subgroup of $P_1(Kb)$ with presentation $\langle a_1, b_1^2 \mid a_1b_1^2 = b_1^2 a_1 \rangle$.
Let $H = p^{-1}(Q)$ be the index 2 subgroup of $P_2(Kb)$.
With this information we have the following commutative diagram of short exact sequences:
$$
\xymatrix{
  &   & 1 \ar[d] & 1 \ar[d] & \\
1 \ar[r] & F_2(a_2,b_2) \ar[r] \ar@{=}[d] & H \ar[r]^{p|_{H}} \ar[d] & Q \ar[r] \ar[d] & 1\\
1 \ar[r] & F_2(a_2,b_2) \ar[r] & P_2(Kb) \ar[r]^{p} \ar[d]^{\psi'} & P_1(Kb) \ar[r] \ar[d]^{\psi} & 1\\
 &   & \Z_2 \ar@{=}[r] \ar[d] & \Z_2  \ar[d] & \\
  &   & 1  & 1  &
  }
$$
where $s\colon P_1(Kb)\to P_2(Kb)$ is a section of $p$ given by $a_1\mapsto a_1a_2$, $b_1\mapsto b_2b_1$,   $\psi\colon P_1(Kb)\to \Z_2$ is given by $a_1\mapsto \overline{0}$ and $b_1\mapsto \overline{1}$, and $\psi'=\psi\circ p$.

By using the method for presentation of extensions, given in \cite[Chapter~10]{J}, and using the presentations of the groups $F_2(a_2,b_2)$ and $Q=\langle a_1, b_1^2 \mid a_1b_1^2 = b_1^2 a_1 \rangle$ we have that the group $H$ has a presentation given by generators $a_2,\, b_2,\, a_1a_2,\, (b_2b_1)^2$ and relations
\begin{itemize}
    \item $(a_1a_2)(b_2b_1)^2 = (b_2b_1)^2(a_1a_2)$
    \item $a_1a_2\cdot a_2\cdot (a_1a_2)^{-1} = a_2$,
    \item $a_1a_2\cdot b_2\cdot (a_1a_2)^{-1} = a_2^{-2}b_2$,
    \item $(b_2b_1)^2\cdot a_2\cdot (b_2b_1)^{-2} = a_2$,
    \item $(b_2b_1)\cdot b_2\cdot (b_2b_1)^{-2} = b_2$.
\end{itemize}

So, by renaming generators, $P_2(Kb)$ has an index two subgroup $H=L\times \mathbb{Z}$ where $L = F_2(a,b)\rtimes \mathbb{Z}$ and  $\mathbb{Z}$ in the semi-direct product is generated by $c$ ($=a_1a_2$) which action is given by $cac^{-1}=a$ and $cbc^{-1}=a^{-2}b$.

Let $N$ be the index two subgroup of $L$ whose quotient is isomorphic to $\setang{a}{a^2=1}$.
Then, by using the Reidemeister-Schreier method (see \cite[Appendix~I.6]{KM}) with the transversal $\Lambda=\{ e, a \}$, we obtain the following presentation for $N$:
\begin{equation}\label{eq:n}
N=\setang{c, b, x, t }{c x c^{-1} = x, ctc^{-1}=x^{-1}t, cbc^{-1}=x^{-1}b}.
\end{equation}

We shall use the presentation of $K$ given by equation~\reqref{k}. Now, we consider the following homomorphisms $\phi\colon K\to N$ and $\psi\colon N\to K$ defined by
$$
\phi(y) = c,\, \phi(\alpha_1)=b\, \textrm{ and }\, \phi(\beta)=x^{-1}
$$
and
$$
\psi(c)=y,\, \psi(b)=\alpha_1,\, \psi(x)=\beta^{-1}\, \textrm{ and }\, \psi(t)=\alpha_1.
$$
Since the composition $\psi\circ \phi\colon K\to K$ is the identity homomorphism (in particular it is injective) then $\phi$ is injective and $K$ is isomorphic to a subgroup of $N$.
As a consequence, from \relem{Knotlerf} and \rethm{lerf1}, $P_2(Kb)$ is not LERF.

\end{proof}

Now, we state the general result for non-large surfaces.

\begin{thm}\label{thm:nlarge}
Let $S$ be a non-large surface.
\begin{enumerate}
	\item $P_2(S)$ is LERF if and only if $S$ is not the Klein bottle.
	\item $P_3(S)$ is LERF if and only if $S$ is either the disk, the sphere or the projective plane.
	\item $P_4(S)$ is LERF if and only if $S$ is the sphere.
	\item $P_n(S)$ is not LERF for every $n\geq 5$.
\end{enumerate}
\end{thm}

\begin{proof}

Let $S$ be a non-large surface. Recall that the case of the disk was covered in \cite{DM}.

\begin{enumerate}
\item
If $S$ is the sphere, projective plane, annulus or torus, then the result follows from Theorem~\ref{thm:LRnlarge}. We know from \reprop{kb} that the group $P_2(Kb)$ is not LERF.
It remains to consider the case of the Möbius band.

Let $S=Mb$ the Möbius band. From \cite[Proof of Proposition A.1, item (a)]{GP} the group $P_2(Mb)$ has a subgroup $G$ of finite index such that $G=F_2\times \mathbb{Z}$, where $F_2$ is the free group of rank 2. The result then follows by applying Theorem \ref{thm:lerf1} and Lemma~\ref{lem:lerf1}.

\item
If $S$ is the sphere then the result follows from Theorem \ref{thm:LRnlarge}.

For the projective plane, using the Fadell-Neuwirth short exact sequence \eqref{eq:ses}, we have that the pure braid group $P_3(\mathbb{R}P^2)$ has $P_2(Mb)$ as a finite index subgroup since $P_2(\mathbb{R}P^2)=Q_8$, where $Q_8$ is the quaternion group.
	Since $P_2(Mb)$ is LERF  (from the first item of this theorem) and has finite index in $P_3(\mathbb{R}P^2)$ we conclude from Theorem~\ref{thm:lerf1} that $P_3(\mathbb{R}P^2)$ is LERF.

Now, let $S$ be the annulus, the torus, the Möbius band or the Klein bottle.
From the short exact sequence \eqref{eq:ses} with $n=1$ and $m=2$ we obtain that $P_3(S)$ has a subgroup $P_2(S\setminus \{x_1\})$ that is not LERF from Theorem~\ref{thm:large}. Therefore, from Theorem~\ref{thm:lerf1} we conclude this item.

\item
In this case, if $S$ is the sphere, the result is also guaranteed by Theorem \ref{thm:LRnlarge}.

Now suppose that $S$ is different from the sphere. From the short exact sequence \eqref{eq:ses} with $n=2$ and $m=2$ for the case of the projective plane, and with $n=1$ and $m=3$ otherwise, we obtain that $P_4(S)$ has a subgroup that is not LERF from Theorem~\ref{thm:large} or item (b) of this theorem. So, applying Theorem~\ref{thm:lerf1} we get the conclusion of this item.

	\item To show $P_l(S)$ is not LERF for every $l\geq 5$ it is enough to consider the short exact sequence \eqref{eq:ses} with $n=3$ and $m=l-3$ and apply Theorem~\ref{thm:lerf1}. Then the respective subgroup is a pure braid group of a large surface $P_{l-3}(S\setminus \{x_1, x_2, x_3\})$ that is not LERF by Theorem~\ref{thm:large}, since $l-3\geq 2$.

\end{enumerate}

\end{proof}

Since (LR) implies LERF then we have the result below.

\begin{cor}\label{cor:LRall}
 Let $S$ be a compact surface. If $S$ is large then $P_n(S)$ is not (LR) for all $n\geq 2$. If $S$ is not large then
 \begin{enumerate}
 \item $P_2(S)$ is (LR) if $S$ is the disk, sphere, projective plane, annulus or torus. $P_2(S)$ is not (LR) if $S$ is the Klein bottle;
	
\item $P_3(S)$ is (LR) if $S$ is the sphere or the disk and not (LR) if $S$ is the annulus, torus, Möbius band or Klein bottle.

	\item $P_4(S)$ is (LR) if and only if $S$ is the sphere.
	\item $P_n(S)$ is not (LR) for every $n\geq 5$.
\end{enumerate}
\end{cor}

\begin{proof}
Follows from Theorem \ref{thm:LRnlarge}, Theorem \ref{thm:large} and Theorem \ref{thm:nlarge}.
\end{proof}

For the pure braids, there are two cases for which we don't know the validity of property (LR): $P_2(S)$ when $S$ is the Möbius band and $P_3(S)$ when $S$ is the projective plane.





For the full braid groups, we have the result below.

\begin{cor}
Let $S$ be a compact surface. If $S$ is large then $B_n(S)$ is not (LR) for all $n\geq 2$. If $S$ is not large then
 \begin{enumerate}
 \item If $S$ is the Klein bottle then $B_2(S)$ is not (LR);
	
\item If $S$ is the annulus, Möbius band, Klein bottle or the torus then $B_3(S)$ is not (LR);

	\item If $S$ is not the sphere then $B_4(S)$ is not (LR);
	\item $B_n(S)$ is not (LR) for every $n\geq 5$.
\end{enumerate}
\end{cor}

\begin{proof}
It follows from  Theorem \ref{thm:large}, Theorem \ref{thm:nlarge} and Corollary \ref{cor:braidpure}.
\end{proof}

It is worth mentioning that   $B_2(S)$ if $S$ is the sphere or projective plane and $B_3(S)$ if $S$ is the sphere are finite groups hence  (LR). If $S$ is the disk then $B_2(S)$ is infinite cyclic hence also (LR). We however don't know the validity of (LR) for: $B_2(S)$ if $S$ is annulus, Möbius band or torus; $B_3(S)$ if $S$ is the projective plane, sphere or disk; $B_4(S)$ if $S$ is the sphere.

\subsection{The generalized word problem}\label{gwp}


The occurrence problem for a finitely presented group $G$ is the problem of deciding, given $w, u_1,\ldots, u_n\in G$ (written as words in generators of $G$), whether $w\in \langle u_1,\ldots, u_n \rangle$, the subgroup of $G$ generated by $u_1,\ldots, u_n$.
Since the occurrence problem has the word problem as a special case (to decide whether $w=1$, ask whether $w\in \langle 1\rangle$), it is also known as the generalized word problem.
The latter term is due to Magnus, who solved the problem for one-relator groups in \cite{Mag}.


Mikhailova \cite{Mik2} showed that the occurrence problem is unsolvable for $F_2\times F_2$. Using this result Makanina \cite{Mak} showed that the occurrence problem is also unsolvable for braid groups $B_n$ with $n\geq 5$ strings and Stillwell \cite{St} showed that the occurrence problem for the mapping class group $M(g,0)$ of the closed orientable surface of genus $g\geq 1$ is solvable if and only if $g=1$.
From Theorem~\ref{thm:f2f2} we may conclude the following result for surface braid groups.

\begin{rem}\label{rem:mak}

Makanina \cite{Mak} did not discuss solvability of the occurrence problem for the pure braid groups of the disk.  However, their construction of a group $F_2\times F_2$ inside $B_n(D^2)$ is such that the free generators of the rank 2 free groups are pure braids. So, from the construction given in \cite{Mak} we also conclude the occurrence problem is unsolvable for pure braid groups $P_n(D^2)$, with $n\geq 5$.
\end{rem}

\begin{prop}

The occurrence problem for the braid group $B_n(S)$ of the surface $S$, and its pure braid subgroup $P_n(S)$, is unsolvable if

\begin{enumerate}
\item $S$ is a surface in  $\mathcal{F}_3$ and $n\geq 2$,

\item the surface $S$ belongs to $\mathcal{F}_2$ and $n\geq 3$,

\item $S$ belongs to $\mathcal{F}_1$, but different from the sphere, and $n\geq 5$,

\item $S$ is the sphere and $n\geq 6$.


\end{enumerate}







\end{prop}

\begin{proof}

The general idea is to exhibit a subgroup $H$ of $P_n(S)$ such that the  $H$ occurrence problem is unsolvable for $H$.
The result follows from Theorem~\ref{thm:f2f2} and the fact that the occurrence problem is unsolvable for $F_2\times F_2$, see \cite{Mik2}.
\end{proof}

 It is known that subgroup separability implies solubility of the generalized word problem  (see \cite[Section 3]{FW} ). Combining that fact with Theorem \ref{thm:nlarge} gives us the following result.

\begin{cor}\label{cor:gwp}
Let $S$ be a non-large surface.  Then the occurence problem for $P_n(S)$ is solvable:
\begin{enumerate}

\item
If $n=2$ and $S$ is not the Klein bottle;

\item
If $n=3$ and $S$ is the disk, sphere or projective plane;

\item
If $n=4$ and $S$ is the sphere.
\end{enumerate}
\end{cor}

We recall that for $n=1$ the (pure) braid group is equal to the fundamental group of $S$, see Remarks~\ref{rem:pi1}.

There are some missing cases, even for the pure braid groups. We don't know the answer for:
\begin{itemize}
\item
$P_2(S)$ if $S\in \mathcal{F}_2$ or if $S$ is the Klein bottle;

\item
$P_n(S)$ if $S$ is the annulus, torus, Möbius Strip or the Klein bottle and $n\in  \{3, 4\}$;

\item
$P_4(S)$ if $S$ is the disk or projective plane.
\end{itemize}

\section{Virtual braid groups}\label{sec:vbg}

The virtual braid group is the natural companion to the category of virtual knots, just as the Artin braid group is to usual knots and links.
We note that a virtual knot diagram is like a classical knot diagram with one extra type of crossing, called a virtual crossing. The virtual braid groups have interpretations in terms of diagrams, see~\cite{Kam2},~\cite{Kau} and~\cite{V}.
The notion of virtual knots and links was introduced by Kauffman together with virtual braids in \cite{Kau}, and since then it has drawn the attention of several researchers.

Let $n\geq 2$ be a positive integer.
 For the first definition, we shall consider the classical presentation of the Artin braid group $B_n$ (see \cite{A2}) and the presentation of the virtual braid group $VB_n$ that was formulated  in \cite[p.798]{V} and restated in \cite{BB}.

\begin{defn}\label{defn:vbn}
The \emph{braid group on $n$ strands}, denoted by $B_n$, is the abstract group generated by $\sigma_i$, for $i=1,2,\dots,n-1$, with the following relations:
		\begin{itemize}
       \item[(AR1)] $\sigma_i\sigma_{i+1}\sigma_{i}=\sigma_{i+1}\sigma_{i}\sigma_{i+1}$, $i=1,2,\dots,n-2;$
       \item[(AR2)] $\sigma_{i}\sigma_j=\sigma_{j}\sigma_i$, $\mid i-j\mid\ge 2$.
\end{itemize}

The \emph{virtual braid group on $n$ strands}, denoted by $VB_n$, is the abstract group generated by $\sigma_i$ (classical generators) and $v_i$ (virtual generators), for $i=1,2,\dots,n-1$, with relations (AR1), (AR2) and:
		\begin{itemize}
       \item[(PR1)] $v_iv_{i+1}v_{i}=v_{i+1}v_{i}v_{i+1}$, $i=1,2,\dots,n-2;$
       \item[(PR2)]  $v_{i}v_j=v_{j}v_i$, $\mid i-j\mid\ge 2;$
			 \item[(PR3)] $v_i^2=1$, $i=1,2,\dots,n-1;$
			 \item[(MR1)]  $\sigma_{i}v_j=v_{j}\sigma_i$, $ \mid i-j\mid\ge 2;$
			 \item[(MR2)] $v_iv_{i+1}\sigma_{i}=\sigma_{i+1}v_{i}v_{i+1}$, $i=1,2,\dots,n-2.$
\end{itemize}
\end{defn}

The generator $\sigma_i$ corresponds to the diagram represented on the left of Figure~\ref{fig:classical}, the generator $\sigma_i^{-1}$ is obtained from $\sigma_i$ by making a crossing change. The geometric generator $v_i$ is illustrated in Figure~\ref{fig:singvirt}.

\begin{figure}[!htb]
\begin{minipage}[b]{0.45\linewidth}
\centering
\includegraphics[width=0.4\columnwidth]{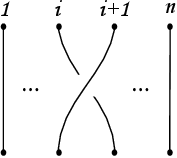}
\caption*{$\sigma_i$}
\end{minipage} \hfill
\begin{minipage}[b]{0.45\linewidth}
\centering
\includegraphics[width=0.4\columnwidth]{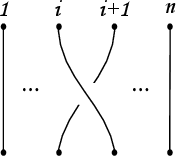}
\caption*{$\sigma_i^{-1}$}
\end{minipage}
\caption{Classical crossings, for $i=1,\ldots, n-1$.}
\label{fig:classical}
\end{figure}

Recently, in \cite{CPM}, Caprau, De la Pena and McGahan introduced virtual singular braids as a generalization of classical singular braids defined by Birman \cite{Bi} and Baez \cite{Ba} for the study of Vassiliev invariants, and virtual braids defined by Kauffman \cite{Kau} and Vershinin \cite{V}. In \cite{CPM} the authors proved an Alexander and Markov Theorem for virtual singular braids and gave two presentations for the monoid of virtual singular braids, denoted by $VSB_n$.
In a more recent paper \cite{CY} Caprau and Yeung showed that the monoid $VSB_n$ embeds in the group $VSG_n$ called the virtual singular braid group on $n$ strands. They also gave a presentation of the virtual singular pure braid group $VSPG_n$ and showed that $VSG_n$ is a semi-direct product of $VSPG_n$ and the symmetric group $S_n$.

Virtual singular braids are similar to classical braids, in addition to having classical crossings, they also have virtual and singular crossings.
The multiplication of two virtual singular braids $\alpha$ and $\beta$ on $n$ strands is given using vertical concatenation. The braid $\alpha\beta$ is formed by placing $\alpha$ on top of $\beta$ and gluing the bottom endpoints of $\alpha$ with the top
endpoints of $\beta$. Under this binary operation, the set of isotopy classes of virtual singular braids on $n$ strands forms a monoid.
Following \cite[Definition~3]{CY} we will call an element in $VSG_n$ an \emph{extended virtual singular braid} or simply a \emph{virtual singular braid} for short.

\begin{figure}[!htb]
\begin{minipage}[b]{0.30\linewidth}
\centering
\includegraphics[width=0.6\columnwidth]{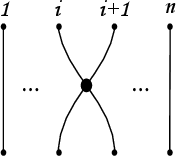}
\caption*{$\tau_i$}
\end{minipage} \hfill
\begin{minipage}[b]{0.30\linewidth}
\centering
\includegraphics[width=0.6\columnwidth]{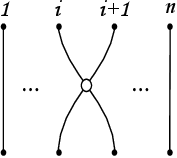}
\caption*{$\tau_i^{-1}$}
\end{minipage} \hfill
\begin{minipage}[b]{0.30\linewidth}
\centering
\includegraphics[width=0.6\columnwidth]{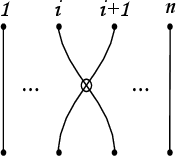}
\caption*{$v_i$}
\end{minipage}
\caption{Singular and virtual crossings, for $i=1,\ldots, n-1$.}
\label{fig:singvirt}
\end{figure}

In \cite[Definition~3]{CY} was defined an abstract group called the virtual singular braid group such that the virtual singular braid monoid $VSB_n$ (defined in \cite{CPM}) embeds in it. We shall use the presentation as stated in \cite[Pages~5-6]{CY}. We note that the relations of the groups $B_n$, $VB_n$ given in their respective presentations (see \redefn{vbn}) appear in the following definition.

\begin{defn}\label{defn:vsgn}
The \emph{virtual singular braid group}, denoted by $VSG_n$, is the abstract group generated by $\sigma_i$ (classical generators), $\tau_i$ (singular generators) and $v_i$ (virtual generators), where $i=1,2,\dots,n-1$, subject to the following relations:
\begin{itemize}
    \item[(2PR)] Two point relations: $v_i^2=1$ and $\sigma_{i}\tau_i=\tau_{i}\sigma_i$, for $i=1,2,\dots,n-1$ for all $i=1,2,\dots,n-1$.
    \item[(3PR)] Three point relations, for all $|i-j|=1$:
		\begin{itemize}
      \item[(3PR1)] $\sigma_i\sigma_{j}\sigma_{i}=\sigma_{j}\sigma_{i}\sigma_{j}$,
      \item[(3PR2)] $v_iv_{j}v_{i}=v_{j}v_{i}v_{j}$,
			\item[(3PR3)] $v_i\sigma_{j}v_{i}=v_{j}\sigma_{i}v_{j}$,
			\item[(3PR4)] $v_i\tau_{j}v_{i}=v_{j}\tau_{i}v_{j}$,
			\item[(3PR5)] $\sigma_i\sigma_{j}\tau_{i}=\tau_{j}\sigma_{i}\sigma_{j}$.
		\end{itemize}
		\item[(CR)] Commuting relations: $g_{i}h_j=h_{j}g_i$ for $\mid i-j\mid\ge 2$, where $g_i, h_i\in \{\sigma_i, \tau_i, v_i \}$.
\end{itemize}
\end{defn}

\begin{rem}\label{rem:embedding}
Let $n\geq 3$. 	From \cite[Proposition~3.1]{CG} and \cite[Theorem~4]{CY} we conclude that the groups $B_n$, $VB_n$ and $SG_n$ are contained in the virtual singular braid group $VSG_n$.
\end{rem}

Finally, we obtain the following result about subgroup separability for virtual braid groups.

\begin{thm}
Let $n\geq 2$.
\begin{enumerate}
\item The virtual braid group $VB_n$ and its pure subgroup $VP_n$ are LERF if and only if $n=2$.

\item The virtual singular braid group $SVG_n$ and its pure subgroup $SVP_n$ are LERF if and only if $n=2$.
\end{enumerate}
\end{thm}

\begin{proof}
Let $n\geq 2$.
\begin{enumerate}

\item
First we consider the case of few strings. We note that there are decompositions of the virtual pure braid groups with few strings as follows
\begin{itemize}
	\item $VB_2=\Z\ast  \Z_2$ (it is obvious from its presentation), and
	\item $VP_3=\overline{P_4}\ast \Z$ (see \cite[Lemma~2.4]{SW}), where $\overline{P_4}$ is isomorphic to the classical pure braid group $P_4$ module its center.
\end{itemize}
Hence $VP_2$ and $VB_2$ are LERF from \rethm{lerf2} and \rethm{lerf1}. On the other hand $VP_3$ is not LERF since $\overline{P_4}$ is not (see \cite[Paragraph before Corollary~1.6]{DM}).
Now we consider the case $n\geq 4$.
It is well known that, for $n\geq 2$, $B_n$ is a subgroup of $VB_n$ (see \cite{Kam} and also \cite[Proposition~3.1]{CG}).
Since $B_n$ is not LERF for $n\geq 4$ \cite[Corollary~1.6]{DM} then $VB_n$ and so $VP_n$ are not LERF for $n\geq 4$.

\item This item holds for $n=2$ from \rethm{lerf2} and the fact that $SVB_2=\Z^2\ast  \Z_2$ (this is clear from \redefn{vsgn}). For $n\geq 3$ it follows from \rerem{embedding} and the first item of this theorem.
\end{enumerate}
\end{proof}

\begin{rem}
Let $n\geq 2$.
Since the Artin braid group $B_n$ is a subgroup of the virtual braid group $VB_n$ and of the virtual singular braid group, see \cite{Kam}, \cite[Proposition~3.1]{CG} and \cite[Remark~4]{O}, then we also may conclude that the occurrence problem is unsolvable for $VB_n$ and $SVG_n$, when $n\geq 5$.
\end{rem}

\end{document}